\documentclass[amsthm,amscd,amssymb,verbatim,epsf,amsmath,amsfonts]{amsart}

\begin{document}
\theoremstyle{plain}
\newtheorem{Thm}{Theorem}
\newtheorem{Cor}{Corollary}
\newtheorem{Con}{Conjecture}
\newtheorem{Main}{Main Theorem}
\newtheorem{Lem}{Lemma}
\newtheorem{Prop}{Proposition}

\theoremstyle{definition}
\newtheorem{Def}{Definition}
\newtheorem{Note}{Note}

\newtheorem{example}{\indent\sc Example}

\theoremstyle{remark}
\newtheorem{notation}{Notation}
\renewcommand{\thenotation}{}

\errorcontextlines=0
\numberwithin{equation}{section}
\renewcommand{\rm}{\normalshape}%

\title[Product of Lorentzian two manifolds]%
   {Lagrangian immersions in the product of Lorentzian two manifolds}
\author{Nikos Georgiou}
\address{Nikos Georgiou\\
          Department of Mathematics and Statistics\\
          University of S\~ao Paulo\\
          Rua do Mat\~ao,  1010 Cidade Universitaria,\\
          S\~ao Paulo,  SP - CEP 05508-090, Brazil.}
\address{Department of Mathematics,\\
          Federal University of S\~ao Carlos,\\
           Rua Washington Luis km 235,\\
          13565-905-S\~ao Carlos -Brazil.}
\email{nikos@ime.usp.br}

\keywords{Lorentzian surfaces, para-Kaehler structure, minimal Lagrangian surfaces, surfaces with parallel mean curvarture vector, Hamiltonian minimal surfaces}
\subjclass{Primary: 51M09; Secondary: 51M30}
\date{25 March 2014}

\begin{abstract}
 For Lorentzian 2-manifolds $(\Sigma_1,g_1)$ and $(\Sigma_2,g_2)$ we consider the two product para-K\"ahler structures $(G^{\epsilon},J,\Omega^{\epsilon})$ defined on the product four manifold $\Sigma_1\times\Sigma_2$, with $\epsilon=\pm 1$. We show that the metric $G^{\epsilon}$ is locally conformally flat (resp. Einstein) if and only if the Gauss curvatures $\kappa_1,\kappa_2$ of $g_1,g_2$, respectively, are both constants satisfying $\kappa_1=-\epsilon\kappa_2$ (resp. $\kappa_1=\epsilon\kappa_2$). We give the conditions on the Gauss curvatures for which every Lagrangian surface with parallel mean curvature vector is the product $\gamma_1\times\gamma_2\subset\Sigma_1\times\Sigma_2$, where $\gamma_1$ and $\gamma_2$ are curves of constant curvature. We study Lagrangian surfaces in the product $d{\mathbb S}^2\times d{\mathbb S}^2$ with non null parallel mean curvature vector and finally, we explore the stability and Hamiltonian stability of certain minimal Lagrangian surfaces and $H$-minimal surfaces.
\end{abstract}

\maketitle

\let\thefootnote\relax\footnote{The author is supported by Fapesp (2010/08669-9)}

\section{Introduction}

This article is a continuation of our previous work \cite{Ge1} on which we have studied minimal Lagrangian surfaces in the K\"ahler structures endowed in the product $\Sigma_1\times \Sigma_2$ of Riemannian two manifolds. Here, we consider again the product structure $\Sigma_1\times \Sigma_2$, where $(\Sigma_1,g_1)$ and $(\Sigma_2,g_2)$ are Lorentzian surfaces. By Lorentz surface we mean a connected, orientable 2-manifold $\Sigma$ endowed with a metric $g$ of indefinite signature. Analogously with the Riemannian case, we may construct a para-K\"ahler structure $(g,j,\omega)$ on a Lorentzian surface $(\Sigma,g)$. We recall that a \emph{para-K\"ahler structure} is a pair $(g,j)$ defined on a manifold $M$ of even dimension and has the same properties than a K\"ahler one, except that $j$ is paracomplex structure rather than complex, i.e,  we have $j^2=Id$ and is integrable, by meaning that the Nijenhuis tensor,
\begin{equation}\label{e:nijenhuis}
N^J(X,Y):=[X,Y]+[JX,JY]-J[JX,Y]-J[X,JY],
\end{equation}
vanishes. Moreover, the compatibilty condition of the metric $g$ with the paracomplex structure $j$ becomes $g(j.,j.)=-g(.,.)$. The symplectic structure $\omega$ can be defined by $\omega(.,.)=g(j.,.)$.

 For Lorentzian surfaces $(\Sigma_1,g_1)$ and $(\Sigma_2,g_2)$, we denote by $\kappa_1$ and $\kappa_2$ the Gauss curvatures of $g_1$ and $g_2$, respectively. For $\epsilon\in\{-1,1\}$, we may define an almost para-K\"ahler structure $(G^{\epsilon},\Omega^\epsilon,J)$ on the product $\Sigma_1\times\Sigma_2$ where $G^{\epsilon}$ is the para-K\"ahler metric, the endomorphism $J$ is an almost paracomplex structure and $\Omega^\epsilon$ is the symplectic 2-form. This structure will be described in Section \ref{s:construction} and in particular we prove:

\begin{Thm}\label{t:einsteinandconformlly} If $(\Sigma_1,g_1)$ and $(\Sigma_2,g_2)$ are Lorentzian two manifolds, the quadruples $(\Sigma_1\times\Sigma_2, G^{\epsilon},J,\Omega^{\epsilon})$ are 4-dimensional para-K\"ahler structures. Furthermore, the  para-K\"ahler metric $G^{\epsilon}$ is conformally flat $($resp.  Einstein$)$ if and only if the Gauss curvatures $\kappa_1$ and $\kappa_2$ are constants with $\kappa_1=-\epsilon \kappa_2$ $($resp. $\kappa_1=\epsilon \kappa_2)$. 
\end{Thm}

In Section \ref{s:rankoneranotwo} we study the surface theory of the para-K\"ahler structures constructed in Section \ref{s:construction}. An analogous result with Theorem 3 of \cite{Ge1} is the following theorem:

\begin{Thm}\label{t:notflattt} Let $(\Sigma_1,g_1)$ and $(\Sigma_2,g_2)$ be Lorentzian two manifolds and let $(G^{\epsilon},J,\Omega^{\epsilon})$ be the para-K\"ahler product structures on $\Sigma_1\times\Sigma_2$ constructed in Section \ref{s:construction}. Assume that one of the following holds:

\emph{(i)} The metrics $g_1$ and $g_2$ are both non-flat almost everywhere and away from flat points we have $\epsilon \kappa_1\kappa_2<0$.

\emph{(ii)} Only one of the metrics $g_1$ and $g_2$ is flat while the other is non-flat almost everywhere. 

\noindent  Then every $\Omega^{\epsilon}$-Lagrangian surface with parallel mean curvature vector is locally the product $\gamma_1\times\gamma_2$, where each curve $\gamma_i\subset\Sigma_i$ has constant curvature.
\end{Thm}

\begin{Note}Note that the Theorem 3 of \cite{Ge1} holds true for Lagrangian immersions with parallel mean curvature vector in $(\Sigma_1\times\Sigma_2, G^{\epsilon},J,\Omega^{\epsilon})$, where $(\Sigma_1,g_1)$ and $(\Sigma_2,g_2)$ are Riemannian two manifolds. \end{Note}

We show that Theorem \ref{t:notflattt} is no longer true when $(\Sigma_1,g_1)$ and $(\Sigma_2,g_2)$ are both flat. In particular, we construct minimal Lagrangian immersions in the paracomplex Euclidean space ${\mathbb D}^2$, endowed with the pseudo-Hermitian product structure, such that they are not a product of straight lines in ${\mathbb D}$  (Proposition \ref{p:prokk}).

If ${\mbox{d}}{\mathbb S}^2$ denotes the anti-De Sitter 2-space, the Theorem \ref{t:notflattt} tells us that $\Omega^-$-Lagrangian surfaces in ${\mbox{d}}{\mathbb S}^2\times {\mbox{d}}{\mathbb S}^2$ with parallel mean curvature vector are locally the product of curves in ${\mbox{d}}{\mathbb S}^2$ with costant curvature. The following theorem proves an analogue result with Theorem 1 in \cite{CU}: 
\begin{Thm}\label{t:desitterw}
Every $\Omega^+$-Lagrangian immersions in ${\mbox{d}}{\mathbb S}^2\times {\mbox{d}}{\mathbb S}^2$, with non null parallel mean curvature vector, is locally the product $(\gamma_1,\gamma_2)$ of curves in ${\mbox{d}}{\mathbb S}^2$ with costant curvatures $k_{1},k_2$, respectively such that $k_1^2+k_2^2>0$.
\end{Thm}

Minimality is the first order condition for a submanifold to be volume-extremizing in its homology class. Minimal submanifolds that are local extremizers of the volume are called \emph{stable minimal submanifolds}. The stability of a minimal submanifold is determined by the monotonicity of the second variation of the volume functional. The second order condition for a minimal submanifold to be volume-extremizing was first derived by Simons \cite{Si} and then Harvey while Lawson have proven that minimal Lagrangian submanifolds of a Calabi-Yau manifold is calibrated, which implies by Stokes theorem, that are volume-extremizing \cite{HL1}. For the stability of minimal Lagrangian surface in $(\Sigma_1\times\Sigma_2, G^{\epsilon},J,\Omega^{\epsilon})$ we prove the following:

\begin{Thm}\label{t:stabilityofminimal} Assume that the Gauss curvatures $\kappa_1$ and $\kappa_2$ satisfy the conditions of Theorem \ref{t:notflattt}.

\emph{(i)} If $\kappa_1$ and $\kappa_2$ are both nonpositive (nonnegative), then every $G^+$-minimal Lagrangian surface $\phi\times\psi$ where the geodesics $\phi,\psi$ are spacelike (timelike), is stable.

\emph{(ii)} If $\kappa_1$ is nonpositive and $\kappa_2$ is nonnegative, then every $G^-$-minimal Lagrangian surface $\phi\times\psi$ where the geodesics $\phi$ is spacelike and $\psi$ is timelike, is stable.
\end{Thm}

A Lagrangian submanifold $\Sigma$ of a (para-)  K\"ahler manifold is said to be \emph{Hamiltonian minimal} (or \emph{$H$-minimal}) if it is a critical point of the volume functional with respect to Hamiltonian variations. A $H$-minimal Lagrangian submanifold is characterized by the fact that its mean curvature vector is divergence-free, that is, ${\mbox{div}}JH=0$, where $J$ is the (para-) complex structure and $H$ is the mean curvature vector of $\Sigma$.  If the second variation of the volume functional of a $H$-minimal submanifold is monotone for any Hamiltonian compactly supported variation, it is said to be \emph{Hamiltonian stable} (or \emph{$H$-stable}). In \cite{Oh1} and \cite{Oh2}, the second variation formula of a $H$-minimal submanifold has been derived in the case of a K\"ahler manifold, while for the pseudo-K\"ahler case it has been given in \cite{AnGer}. The next theorem, in Section \ref{s:hamiltstabilitysection}, investigates the $H$-stability of projected rank one Hamiltonian $G^{\epsilon}$-minimal surfaces in $\Sigma_1\times\Sigma_2$:

\begin{Thm}\label{t:hstability} Let $\Phi=(\phi,\psi)$ be of projected rank one Hamiltonian $G^{\epsilon}$-minimal immersion in $(\Sigma_1\times\Sigma_2,G^{\epsilon})$ such that $\epsilon_{\phi}\kappa_1\leq 0$ and $\epsilon_{\psi}\kappa_2\leq 0$ along the curves $\phi$ and $\psi$ respectively. Then $\Phi$ is a local maximizer of the volume in its Hamiltonian isotopy class. \end{Thm}

\vspace{0.1in}

\noindent {\bf Acknowledgements.} The author would like to thank H. Anciaux, B. Guilfoyle and W. Klingenberg for their helpful and valuable suggestions and comments.

\vspace{0.1in}

\section{The Product para-K\"ahler structure}\label{s:construction}

Let $(\Sigma,g)$ be a two dimensional oriented manifold endowed with a non degenerate Lorentzian metric $g$. Then in a neighbourhood of any point there exist local isothermic coordinates $(s,t)$, i.e., $g_{ss}=-g_{tt}$ and $g_{st}=0$ (see \cite{An2}).
The endomorphism $j:{\mbox{T}}\Sigma\rightarrow {\mbox{T}}\Sigma$ defined by $j(\partial/\partial s)=\partial/\partial t$ and $j(\partial/\partial t)=\partial/\partial s$, satisfies $j^2={\mbox{Id}}_{T\Sigma}$ and $g(j.,j.)=-g(.,.)$. It follows that $j$ defines a paracomplex structure and if we set $\omega(\cdot, \cdot)=g(j_k\cdot, \cdot)$, the quadruple $(\Sigma,g,j,\omega)$ is a 2-dimensional para-K\"ahler manifold. 

For $k=1,2$, let $(\Sigma_k,g_k,j_k,\omega_k)$  be the para-K\"ahler structures defined as before, and consider the product structure $\Sigma_1\times\Sigma_2$. The identification $X\in {\mbox{T}}(\Sigma_1\times\Sigma_2)\simeq (X_1,X_2)\in {\mbox{T}}\Sigma_1\oplus {\mbox{T}}\Sigma_2$, gives the natural splitting ${\mbox{T}}(\Sigma_1\times\Sigma_2)={\mbox{T}}\Sigma_1\oplus {\mbox{T}}\Sigma_2$. For $(x,y)\in \Sigma_1\times\Sigma_2$, let $X=(X_1,X_2)$ and $Y=(Y_1,Y_2)$ be two tangent vectors in ${\mbox{T}}_{(x,y)}(\Sigma_1\times\Sigma_2)$ and define the metric $G^{\epsilon}$ by
\[
G^{\epsilon}_{(x,y)}(X,Y)=g_1(X_1,Y_1)(x)+\epsilon g_2(X_2,Y_2)(y),
\]
where $\epsilon\in\{-1,1\}$. If $\nabla$ denotes the Levi-Civita connection with respect to the metric $G^{\epsilon}$, we have
$\nabla_X Y=(D^1_{X_1} Y_1,D^2_{X_2} Y_2)$, where $D^1,D^2$ denote the Levi-Civita connections with respect to the metrics $g_1$ and $g_2$, respectively. The endomorphism $J\in {\mbox{End}}({\mbox{T}}\Sigma_1\oplus {\mbox T}\Sigma_2)$ defined by $J=j_1\oplus j_2$ is an almost paracomplex structure on $\Sigma_1\times\Sigma_2$, while the two-forms 
$\Omega^{\epsilon}=\pi_1^{\ast}\omega_1+\epsilon\pi_2^{\ast}\omega_2$,
 are symplectic structures, with $\pi_i:\Sigma_1\times \Sigma_2\rightarrow \Sigma_i$ is the $i$-th projection.

\vspace{0.1in}

{\indent\sc Proof of Theorem \ref{t:einsteinandconformlly}:}
It is clear that the Nijenhuis tensor of $J$ given by (\ref{e:nijenhuis}) vanishes. Furthermore, $J$ and $\Omega^{\epsilon}$ are compatible, i.e., $\Omega^{\epsilon}(J.,J.)=-\Omega^{\epsilon}(.,.)$ and thus the quadruples $(\Sigma_1\times\Sigma_2, G^{\epsilon},J,\Omega^{\epsilon})$  are para-K\"ahler structures.

Let $(e_1,e_2)$ and $(v_1,v_2)$ be orthonormal frames on $\Sigma_1$ and $\Sigma_2$ respectively, both oriented such that $|e_1|^2=|v_1|^2=1$ and $|e_2|^2=|v_2|^2=-1$ and consider an  orthonormal frame $(E_1,E_2,E_3, E_4)$ of $G^{\epsilon}$ defined by 
\[
E_1=(e_1,v_1+v_2),\quad E_2=(e_2,v_1+v_2),
\quad 
 E_3=(\epsilon(e_2-e_1),v_1),\quad E_4=(\epsilon(e_1-e_2),v_2).
\]
The Ricci curvature tensor ${\mbox{Ric}}^{\epsilon}$ of $G^{\epsilon}$ gives
\[
{\mbox Ric}^{\epsilon}(E_1,E_1)_{(x,y)}=\kappa_1(x),\qquad
{\mbox Ric}^{\epsilon}(E_3,E_3)_{(x,y)}=\kappa_2(y),
\]
and using the fact ${\mbox Ric}^{\epsilon}(J.,J.)=-{\mbox Ric}^{\epsilon}(.,.)$, the scalar curvatute ${\mbox R}^\epsilon$ is:
\begin{equation}\label{e:scpos}
{\mbox  R}^\epsilon=2(\kappa_1(x)+\epsilon\kappa_2(y)).
\end{equation}
If $G^{\epsilon}$ is conformally flat, it is scalar flat \cite{AR} and thus, from (\ref{e:scpos}), the Gauss curvatures $\kappa_1, \kappa_2$ are constants with $\kappa_1=-\epsilon\kappa_2$.

Conversely, assuming that $\kappa_1=-\epsilon\kappa_2=c$, where $c$ is a real constant and following the same computations with the proof of Theorem 2.2 in \cite{Ge1}, we prove that the self-dual ${\mbox  W}^+$ and the anti-self-dual part ${\mbox W}^-$ of the Weyl tensor vanish and therefore the metric $G^\epsilon$ is conformally flat.

On the other hand, a direct computation shows that ${\mbox  Ric}^\epsilon(E_i,E_j)=cG^\epsilon(E_i,E_j)$ if and only if $\kappa_1=\epsilon\kappa_2=c$, where $c$ is a real constant and the theorem follows.
$\Box$ 

\begin{Cor}
Let $(\Sigma,g)$ be a Lorentzian two manifold. The para-K\"ahler metric $G^-$ \emph{(}resp. $G^+$\emph{)} of the four dimensional K\"ahler manifold $\Sigma\times\Sigma$ is conformally flat  \emph{(}resp. Einstein\emph{)} if and only if the metric $g$ is of constant Gaussian curvature.
\end{Cor}

\vspace{0.1in}

\section{Lagrangian immersions in $\Sigma_1\times\Sigma_2$}\label{s:rankoneranotwo}

In this section, we study Lagrangian immersions in the product $\Sigma_1\times\Sigma_2$ endowed with the para-k\"ahler structure $(G^\epsilon,J,\Omega^\epsilon)$ constructed in section \ref{s:construction}. An immersion $\Phi:S\rightarrow\Sigma_1\times\Sigma_2$ of a surface $S$ is said to be Lagrangian if $\Phi^{\ast}\Omega^\epsilon$ vanishes for every point of $S$. In this case, the paracomplex structure $J:{\mbox T}S\rightarrow {\mbox N}S$ is a bundle isomorphism between the tangent bundle ${\mbox T}S$ and the normal bundle ${\mbox N}S$. It is well known that a Lagrangian immersion of a pseudo-Riemannian K\"ahler manifold is indefinite if the K\"ahler metric is indefinite. Althought the signature of the para-K\"ahler metric is always neutral, a Lagrangian immersion can be either Riemannian or indefinite. If $\pi_i$ are the projections of $\Sigma_1\times\Sigma_2$ onto $\Sigma_i$, $i=1,2$, we denote by $\phi$ and $\psi$ the mappings $\pi_1\circ\Phi$ and $\pi_2\circ\Phi$, respectively, and we write $\Phi=(\phi,\psi)$.
\begin{Def}\label{d:defiprojectedrank}
The immersion $\Phi=(\phi,\psi): S\rightarrow\Sigma_1\times\Sigma_2$ is said to be of \emph{projected rank zero} at a point $p\in S$ if either ${\mbox rank}(\phi(p))=0$ or ${\mbox rank}(\psi(p))=0$. $\Phi$ is of \emph{projected rank one} at $p$ if either ${\mbox rank}(\phi(p))=1$ or ${\mbox rank}(\psi(p))=1$. Finally, $\Phi$ is of \emph{projected rank two} at $p$ if ${\mbox rank}(\phi(p))={\mbox rank}(\psi(p))=2$. 
\end{Def}
Note that the fact that $\Phi$ is an immersion, implies that $\Phi$ is locally either of projected rank zero, one or two.

\vspace{0.1in}


Let $\Phi=(\phi,\psi)$ be of projected rank zero immersion in $\Sigma_1\times\Sigma_2$. Assuming, without loss of generality, that ${\mbox rank}(\phi)=0$, the map $\phi$ is locally a constant function and the map $\psi$ is a local diffeomorphism. Following the same argument with Proposition 3.2 in \cite{Ge1}, we show that there are no Lagrangian immersions in $\Sigma_1\times\Sigma_2$ of projected rank zero.

\vspace{0.1in}


In order to discuss Lagrangian surfaces of projected rank one, we need to extend the definition of Cornu spirals for  a pseudo-Riemannian two manifold.

\begin{Def}\label{d:cornuspiral}
Let $(\Sigma,g)$ be a  pseudo-Riemannian two manifold. A non-null regular curve $\gamma$ of $\Sigma$ is called a \emph{Cornu spiral of parameter $\lambda$} if its curvature $\kappa_{\gamma}$ is a linear function of its arclength parameter such that $\kappa_{\gamma}(s)=\lambda s+\mu$, where $s$ is the arclength and $\lambda,\mu$ are real constants.  
\end{Def}

\begin{Prop}\label{t:rankonethe}
Let $(\Sigma_1,g_1)$ and $(\Sigma_2,g_2)$ be Lorentzian surfaces and $\Phi$ be a nondegenerate $\Omega^\epsilon$-Lagrangian surface of projected rank one. Then,

\noindent 1) $\Phi$ can be locally parametrised by $\Phi:S\rightarrow \Sigma_1\times\Sigma_2:(s,t)\mapsto (\phi(s),\psi(t))$, where $\phi$ and $\psi$ are respectively non-null regular curves in $\Sigma_1$ and $\Sigma_1$ with $s,t$ being their arclength parameters,

\noindent 2) the induced metric $\Phi^{\ast}G^{\epsilon}$ is flat,

\noindent 3) $\Phi$ is $G^{\epsilon}$-minimal if and only if the curves $\phi$ and $\psi$ are geodesics,

\noindent 4) $\Phi$ is $H$-minimal if and only if $\phi$ and $\psi$ are Cornu spirals of parameters $\lambda_{\phi}$ and $\lambda_{\psi}$, respectively, such that
\begin{equation}\label{e:condhamilrankonw}\epsilon_{\phi}\lambda_{\phi}+\epsilon\epsilon_{\psi}\lambda_{\psi}=0,\end{equation}
where $\epsilon_{\phi}=g_1(\phi',\phi')$ and $\epsilon_{\psi}=g_2(\psi',\psi')$. 
\end{Prop} 
\begin{proof} Let $\Phi=(\phi,\psi):S\rightarrow\Sigma_1\times\Sigma_2$ be of  projected rank one Lagrangian immersion. Then either $\phi$ or $\psi$ is of rank one. Assume, without loss of generality, that $\phi$ is of rank one. The nondegeneracy of $\omega_2$ implies that $\psi$ is of rank one and thus $S$ is locally parametrised by $\Phi:U\subset S\rightarrow\Sigma_1\times\Sigma_2:(s,t)\mapsto (\phi (s),\psi (t))$, where $\phi$ and $\psi$ are regular curves in $\Sigma_1$ and $\Sigma_2$, respectively. We denote by $s,t$ the arclength parameters of $\phi$ and $\psi$, respectively, such that $g_1(\phi',\phi')=\epsilon_{\phi}$ and $g_s(\psi',\psi')=\epsilon_{\psi}$, where $\epsilon_{\phi},\epsilon_{\psi}\in\{-1,1\}$.
The Fr\'enet equations give
\[
D^1_{\phi'}\phi'=k_{\phi}j_1\phi'\qquad D^2_{\psi'}\psi'=k_{\psi}j_2\psi',
\]
where $k_{\phi}$ and $k_{\psi}$ are the curvatures of $\phi$ and $\psi$, respectively. Moreover, $\Phi_s=(\phi',0)$ and $\Phi_t=(0,\psi')$ and thus
\[
\nabla_{\Phi_s}\Phi_s=(k_{\phi}j_1\phi',0),\qquad \nabla_{\Phi_t}\Phi_t=(0,k_{\psi}j_2\psi'),\qquad \nabla_{\Phi_t}\Phi_s=(0,0).
\] 
The immersion $\Phi$ is flat since the first fundamental form $G^{\epsilon}_{ij}=G^{\epsilon}(\partial_i\Phi,\partial_j\Phi)$ is given by $G_{ss}=\epsilon_{\phi},\; G_{tt}=\epsilon\epsilon_{\psi}$, and $G_{st}=0$. The second fundamental form $h^{\epsilon}$ of $\Phi$ is completely determined by the following tri-symmetric tensor 
\[
h^{\epsilon}(X,Y,Z):=G^{\epsilon}(h^{\epsilon}(X,Y),JZ)=\Omega^{\epsilon}(X,\nabla_Y Z).
\]
We then have $h^{\epsilon}_{sst}=h^{\epsilon}_{stt}=0,\; h_{sss}^{\epsilon}=-\epsilon_{\phi}k_{\phi} $ and $h_{ttt}^{\epsilon}=-\epsilon\epsilon_{\psi} k_{\psi}$. Denoting by $H^{\epsilon}$ the mean curvature of $\Phi$ we have
\begin{equation}\label{e:meancurvlll}
2H^{\epsilon}=\epsilon_{\phi}k_{\phi}J\Phi_s+\epsilon \epsilon_{\psi}k_{\psi}J\Phi_t,
\end{equation}
and we can see easily that the Lagrangian immersion $\Phi$ is $G^{\epsilon}$-minimal if and only if the curves $\phi$ and $\psi$ are geodesics. Moreover, if $\Phi$ is a $G^{\epsilon}$-minimal Lagrangian it is totally geodesic, since the second fundamental form vanishes identically. 

The condition (\ref{e:condhamilrankonw}) for Hamiltonian $G^\epsilon$-minimal Lagrangian surfaces of projected rank one, is given by the fact that
\[
\mbox{div}^{\epsilon}(2JH^{\epsilon})=\epsilon_{\phi}\frac{D}{ds}k_{\phi}(s)+\epsilon\epsilon_{\psi} \frac{D}{dt}k_{\psi}(t),\]
and the Proposition follows.
\end{proof}

We now prove our next theorem:

\vspace{0.1in}

{\indent\sc Proof of Theorem \ref{t:notflattt}:}
Let $\Phi=(\phi,\psi):S\rightarrow \Sigma_1\times\Sigma_2$ be a $\Omega^{\epsilon}$-Lagrangian immersion of  projected rank two. Then the  mappings $\phi:S\rightarrow \Sigma_1$ and $\psi:S\rightarrow \Sigma_2$ are both local diffeomorphisms. The Lagrangian condition yields 
\begin{equation}\label{e:lagpositive}
\phi^{\ast}\omega_1=-\epsilon\psi^{\ast}\omega_2.
\end{equation}
Without loss of generality we consider an orthonormal frame $(e_1,e_2)$ of $\Phi^{\ast}G^{\epsilon}$ such that,
\[
G^{\epsilon}(d\Phi(e_1),d\Phi(e_1))=\epsilon_1 G^{\epsilon}(d\Phi(e_2),d\Phi(e_2))=1,\qquad G^{\epsilon}(d\Phi(e_1),d\Phi(e_2))=0.
\]
Note that for $\epsilon_1=1$ the induced metric $\Phi^{\ast}G^{\epsilon}$ is Riemannian while for $\epsilon_1=-1$ the induced metric $\Phi^{\ast}G^{\epsilon}$ is Lorentzian.

Let $(s_1,s_2)$ and $(v_1,v_2)$ be oriented orthonormal frames of $(\Sigma_1,g_1)$ and $(\Sigma_2,g_2)$, respectively, such that $|s_1|^2=-|s_2|^2=1$ and $|v_1|^2=-|v_2|^2=1$ with $j_1s_1=s_2$ and $j_2v_1=v_2$. Then there exist smooth functions $\lambda_1,\lambda_2,\mu_1,\mu_2$ on $\Sigma_1$ and $\bar\lambda_1,\bar\lambda_2,\bar\mu_1,\bar\mu_2$ on $\Sigma_2$ such that
\[
d\phi(e_1)=\lambda_1 s_1+\lambda_2 s_2,\;\; d\phi(e_2)=\mu_1 s_1+\mu_2 s_2,
\;\;
d\psi(e_1)=\bar\lambda_1 v_1+\bar\lambda_2 v_2,\;\; d\psi(e_2)=\bar\mu_1 v_1+\bar\mu_2 v_2.
\]
Using the Lagrangian condition (\ref{e:lagpositive}), we have
\[
(\lambda_1\mu_2-\lambda_2\mu_1)(\phi(p))=-\epsilon(\bar\lambda_1\bar\mu_2-\bar\lambda_2\bar\mu_1)(\psi(p)),\quad\forall\;\; p\in S.
\]
Moreover, the assumption that $\Phi$ is of projected rank two, implies that $\lambda_1\mu_2-\lambda_2\mu_1\neq 0$ for every $p\in S$.

For the mean curvature vector $H^{\epsilon}$  of the immersion $\Phi$, consider the one form $a_{H^{\epsilon}}$ defined by $a_{H^{\epsilon}}=G^{\epsilon}(JH^{\epsilon},\cdot)$. Since $\Phi$ is a Lagrangian in a para-K\"ahler 4-manifold
\[
da_{H^{\epsilon}}=-\Phi^{\ast}\rho^{\epsilon},
\] 
where $\rho^{\epsilon}$ is the Ricci form of $G^{\epsilon}$. The fact that $\Phi$ has parallel mean curvature vector implies that the one form $a_{H^{\epsilon}}$ is closed and thus,
\[
(\mu_1\lambda_2-\mu_2\lambda_1)\Big[\Big(\lambda_1^2-\lambda_2^2+\epsilon_1(\mu_1^2-\mu_2^2)\Big)\kappa_1-
\Big(\bar\lambda_1^2-\bar\lambda_2^2+\epsilon_1(\bar\mu_1^2-\bar\mu_2^2)\Big)\kappa_2\Big)\Big]=0.
\]
Hence,
\begin{equation}\label{e:consd}
\Big(\lambda_1^2-\lambda_2^2+\epsilon_1(\mu_1^2-\mu_2^2)\Big)\kappa_1=
\Big(\bar\lambda_1^2-\bar\lambda_2^2+\epsilon_1(\bar\mu_1^2-\bar\mu_2^2)\Big)\kappa_2.
\end{equation}
Following the same argument with the proof of Theorem 3.5 in \cite{Ge1}, we show that
\begin{equation}\label{e:niequat}
\lambda_1^2-\lambda_2^2+\epsilon_1(\mu_1^2-\mu_2^2)=1, \quad\mbox{and}\quad \bar\lambda_1^2-\bar\lambda_2^2+\epsilon_1(\bar\mu_1^2-\bar\mu_2^2)=\epsilon,
\end{equation}
and the relation (\ref{e:consd}) becomes
\[
\kappa_1(\phi(p))=\epsilon \kappa_2(\psi(p)),\qquad \mbox{for}\;\mbox{every}\;\; p\in S,
\]
which implies that the metrics $g_1$ and $g_2$ can satisfy neither condition (i) nor condition (ii) of the statement and the theorem follows.
$\Box$

\vspace{0.1in}

We immediately obtain the following corollary:

\begin{Cor}\label{c:coriiwu}
Let $(\Sigma,g)$ be a non-flat Lorentzian two manifold. Then every $G^{-}$-minimal Lagrangian surface immersed in $\Sigma\times\Sigma$ is of projected rank one and consequently the product of two geodesics of $(\Sigma,g)$.
\end{Cor}

\begin{example}Consider the real space ${\mathbb R}^3$, endowed with the pseudo-Riemannian metric $\left<.,.\right>_p=-\sum_{i=1}^p dx_i^2+\sum_{i=p+1}^3 dx_i^2$. We define the \emph{de Sitter 2-space}  ${\mbox d}{\mathbb S}_a^2$ and the \emph{anti de Sitter 2-space} ${\mbox Ad}{\mathbb S}^2_a$ of radius $a>0$, by ${\mbox d}{\mathbb S}_a^2:=\{x\in {\mathbb R}^3|\; \left<x,x\right>_1=a^2\}$ and ${\mbox Ad}{\mathbb S}^2_a:=\{x\in {\mathbb R}^3|\; \left<x,x\right>_2=a^2\}$. Note that ${\mbox d}{\mathbb S}_a^2$ and ${\mbox Ad}{\mathbb S}^2_a$ are anti-isometric and therefore we only use  the De Sitter 2-space ${\mbox d}{\mathbb S}_a^2$. Moreover, the Gauss curvature is constant with $\kappa({\mbox d}{\mathbb S}_a^2)=a^{-1}$. For positives $a,b$ with $a\neq b$, the Theorem \ref{t:notflattt} implies that every $\Omega^\epsilon$-Lagrangian surface with parallel mean curvature in ${\mbox d}{\mathbb S}_a^2\times {\mbox d}{\mathbb S}_b^2$ is locally the product of geodesics $\gamma_1\times\gamma_2$. \end{example}

\begin{example}\label{e:gaussmaporgeodesiccongruence} The space $L^{-}({\mbox Ad}{\mathbb S}^3)$ of oriented timelike geodesics in the anti-De Sitter 3-space ${\mbox Ad}{\mathbb S}^3$ is diffeomorphic to the product ${\mbox d}{\mathbb S}^2\times {\mbox d}{\mathbb S}^2$. The para-K\"ahler metric $G^\epsilon$ is invariant under the natural action of the isometry group ${\mbox Iso}({\mbox Ad}{\mathbb S}^3,g)$ (see \cite{AGK} and \cite{An4}). It is clear by Corollary \ref{c:coriiwu} that every $G^{-}$-minimal Lagrangian surface immersed in $L^{-}({\mbox Ad}{\mathbb S}^3)$ is locally the product of two geodesics in ${\mbox d}{\mathbb S}^2$. \end{example}

\vspace{0.1in}

The set of para-complex numbers ${\mathbb D}$ is defined to be the two-dimensional real vector space ${\mathbb R}^2$ endowed with the commutative algebra structure whose product rule is 
\[
(x_1,y_1)\cdot (x_2,y_2)=(x_1x_2+y_1y_2,x_1y_2+x_2y_1).
\]

The number $(0,1)$, whose square is $(1,0)$, will be denoted by $\tau.$
It is convenient to use the following notation: $(x,y) \simeq z =x + \tau y .$
In particular, one has the same conjugation operator than in ${\mathbb C}$, i.e., $\overline{x+ \tau y } = x - \tau y$ with corresponding square norm $| z |^2 :=  z . \bar{z} = x^2 -y^2$.
In other words, the metric associated to  $|.|^2$ is the Minkowski metric $dx^2 - dy^2.$

On the Cartesian product ${\mathbb D}^2$ with para-complex coordinates $(z_1,z_2)$, we define the canonical  para-K\"ahler structure $(J,G^\epsilon)$ by
\[
J(z_1,z_2):=(\tau z_1,\tau z_2),
\qquad
G^\epsilon:=dz_1  d\bar{z}_1+\epsilon dz_2  d\bar{z}_2.
\]

We now give a similar result of Dong \cite{Dong} for Lagrangian graphs in ${\mathbb C}^n$, to show that there are $G^-$-minimal Lagrangian surface of rank two in $\Sigma\times\Sigma$ when $(\Sigma,g)$ is a flat Lorentzian surface. 

\begin{Prop}\label{p:prokk}
Let $U$ be some open subset of ${\mathbb R}^2$, $u$ a smooth, real-valued function defined on $U$ and $\Phi$ be the graph of the gradient i.e. the immersion $\Phi:U\rightarrow {\mathbb D}^2:x\mapsto x+\tau\nabla u(x)$, where $\nabla u$ is the gradient of $u$ with respect to the Minkowski metric $g$ of ${\mathbb R}^2$. Then $\Phi$ is $G^-$-minimal Lagrangian if and only if $u$ is a harmonic function with respect to the metric $g$.
\end{Prop}
\begin{proof}
If $x=(x_1,x_2)$, the first derivatives of the immersion $\Phi(x_1,x_2)=(x_1+\tau u_{x_1},x_2-\tau u_{x_2})$ are
\[
\Phi_{x_1}=(1+\tau u_{x_1x_1},-\tau u_{x_1x_2}),\qquad \Phi_{x_2}=(\tau u_{x_1x_2},1-\tau u_{x_2x_2}).
\]
The symplectic structure is $\Omega^{-}(.,.)=G^{-}(J.,.)$. Then $\Omega^{-}(\Phi_{x_1},\Phi_{x_2})=0$ and therefore $\Phi$ is a Lagrangian immersion. The Lagrangian angle is
\begin{eqnarray}
\beta&=&{\mbox arg}(1-u_{x_1x_1}u_{x_2x_2}+u_{x_1x_2}^2+\tau(u_{x_1x_1}-u_{x_2x_2})).\nonumber
\end{eqnarray}
Then $\Phi$ is $G^{-}$-minimal if and only if $u_{x_1x_1}-u_{x_2x_2}=0$ and the proposition follows. 
\end{proof}

\noindent
The immersion $\Phi$ in the Proposition \ref{p:prokk} is of rank two at the open subset $\{(x_1,x_2)\in U|\; u_{x_1x_2}\neq 0\}$.

\vspace{0.1in}

We denote by $d{\mathbb S}^2$ the De Sitter 2-space of radius one. Then we prove our next result:

\vspace{0.1in}

{\indent\sc Proof of Theorem \ref{t:desitterw}:}
If ${\mathbb R}^{1,2}$ denotes the Lorentzian space $({\mathbb R}^3,\left<,\right>_1)$, we define the Lorentzian cross product $\otimes$ in ${\mathbb R}^{1,2}$ by
\[
 u\otimes v:={\mbox I}_{1,2}\cdot (u\times v), 
\]
where $u\times v$ is the standard cross product in ${\mathbb R}^3$ and ${\mbox I}_{1,2}={\mbox diag}(-1,1,1)$. For $u,v,w\in {\mathbb R}^3$ we have
\begin{equation}\label{e:lorcrosspro}
\left<u\otimes v,u\otimes w\right>_1=-\left<u,u\right>_1\left<v,w\right>_1+\left<u,v\right>_1\left<u,w\right>_1.
\end{equation}
The paracomplex structure $j$ on ${\mbox d}{\mathbb S}^2$ is given by $j_x(v):=x\otimes v$, where $v\in {\mathbb R}^3$ is such that $\left<x,v\right>_1=0$. It can be verified easily that $x\otimes (x\otimes v)=v$.

If $h$ denotes the second fundamental form of the inclusion map $i:{\mbox d}{\mathbb S}^2\hookrightarrow {\mathbb R}^3$ and $u,v\in {\mbox T}_x{\mbox d}{\mathbb S}^2$, we have $h_x(u,v)=-\left<u,v\right>_1x$. We consider the para-K\"ahler structure $(G^+,\Omega^+,J)$ of the product ${\mbox d}{\mathbb S}^2\times {\mbox d}{\mathbb S}^2$ and denote by $\tilde{h}$ the second fundamental form of ${\mbox d}{\mathbb S}^2\times {\mbox d}{\mathbb S}^2$ into ${\mathbb R}^3\times {\mathbb R}^3$. Thus,
\[
\tilde{h}_{(x,y)}(U,V)=(-\left<u_1,v_1\right>_1x,-\left<u_2,v_2\right>_1y),
\]
where $U=(u_1,u_2), V=(v_1,v_2)\in {\mbox T}_{(x,y)}({\mbox d}{\mathbb S}^2\times {\mbox d}{\mathbb S}^2)$. From (\ref{e:lorcrosspro}) we obtain
\begin{equation}\label{e:lorcrosspro1}
\tilde{h}_{(x,y)}(JU,JV)=-\tilde{h}_{(x,y)}(U,V).
\end{equation}
For an orthonormal frame $(v_1,v_2)$ of ${\mbox d}{\mathbb S}^2$, oriented such that $|v_1|^2=-|v_2|^2=1$, we consider the following oriented orthonormal frame $(E_1,E_2=JE_1,E_3,E_4=JE_3)$ of $G^+$ defined by
\[
E_1=(v_1,v_1+v_2),\qquad E_3=(-v_1+v_2,v_1),
\]
and we prove that the mean curvature $\tilde{H}$ of the inclusion map of ${\mbox d}{\mathbb S}^2\times {\mbox d}{\mathbb S}^2$ into ${\mathbb R}^3\times {\mathbb R}^3$ is
\begin{equation}\label{e:lorcrosspro2}
2\tilde{H}_{(x,y)}=-(x,y).
\end{equation}
Let $\Phi=(\phi,\psi):S\rightarrow {\mbox d}{\mathbb S}^2\times {\mbox d}{\mathbb S}^2$ be a $\Omega^+$- Lagrangian immersion with non null parallel mean curvature vector $H$. Following similar arguments with Theorem 1 of \cite{CU}, consider an orthonormal frame $(e_1,e_2)$ with respect to the induced metric such that $|e_1|^2=\epsilon_1|e_2|^2=1$. Then the equation (\ref{e:niequat}) becomes,
\begin{equation}\label{e:niequat1}
|d\phi(e_1)|^2+\epsilon_1 |d\phi(e_2)|^2=|d\psi(e_1)|^2+\epsilon_1 |d\psi(e_2)|^2=1.
\end{equation}
By the proof of Theorem \ref{t:notflattt}, we have
\begin{equation}\label{e:hsgrwu}
(|d\phi e_1|^2-\epsilon_1|d\phi e_2|^2)^2+4\epsilon_1g_1(d\phi e_1,d\phi e_2)^2=1+4\epsilon_1 C^2,
\end{equation}
where $C:=\lambda_2\mu_1-\lambda_1\mu_2=\bar\lambda_1\bar\mu_2-\bar\lambda_2\bar\mu_1$ and is called \emph{the associated Jacobian} of the Lagrangian immersion $\Phi$ (see \cite{CU} and \cite{urbano2} for definitions and further details). Note that the vanishing of the associated Jacobian is equivalent with the fact that the Lagrangian immersion $\Phi$ is of projected rank one. 

From (\ref{e:lorcrosspro1}) and  (\ref{e:lorcrosspro2}), the mean curvature vector $\bar H$ of $S$ in ${\mathbb R}^3\times {\mathbb R}^3$ is
\begin{equation}\label{e:fundatla}
\bar H= H-\frac{1}{2}\Phi.
\end{equation}
Since $\nabla ^{\bot}H=0$ together with the Lagrangian condition, implies that $\nabla JH=0$, where $\nabla ^{\bot}$ and $\nabla$ denote the normal and tangential part of the Levi-Civita connection of $G^+$. From Theorem \ref{t:einsteinandconformlly}, we know that $G^+$ is Einstein and so there exists locally a function $\beta$ on $S$ such that $JH=\nabla\beta$. Thus, using the Boschner formula,
\[
\frac{1}{2}\Delta|\nabla\beta|^2={\mbox Ric}(\nabla\beta,\nabla\beta)+g(\nabla\beta,\nabla\Delta\beta)+g(\nabla^2\beta,\nabla^2\beta),
\]
we have that the induced metric $g$ is flat. Furthermore, the normal curvature of $\Phi$ vanishes. It is important to mention that the Boschner formula holds also for pseudo-Riemannian metrics \cite{AnGer}.

Let $(x,y)$ be isothermal local coordinates of $g$, i.e., $g=e^{2u}(dx^2+\epsilon_1 dy^2)$ and let $z=x+iy$. Note that for $\epsilon_1=1$, the variable $z$ is a local holomorphic coordinate, while for $\epsilon_1=-1$, the variable $z$ is a local paraholomorphic coordinate. A brief computation gives, 
\[
G^+(\Phi_z,\Phi_z)=g_1(\phi_z,\phi_z)+g_2(\psi_z,\psi_z)=0
\]
From (\ref{e:fundatla}) we have that $\Phi_{z\bar z}=e^{2u}\Big(H-\Phi/2\Big)/2$, and
\[
\Phi_{zz}=2u_z\Phi_z-G^+(H,J\Phi_z)J\Phi_z-2G^+(\Phi_{zz},J\Phi_z)J\Phi_{\bar z}-\frac{1}{2}G^+(\Phi_z,\hat{\Phi}_z)\hat{\Phi},
\]
where $\hat{\Phi}=(\phi,-\psi)$. A direct computation implies,
\[
G^+(J\Phi_{\bar z},\hat{\Phi}_z)=i\epsilon_1 e^{2u}C,
\]
and thus, $\hat{\Phi}_z=2G^+(\Phi_z,\hat{\Phi}_z)\Phi_{\bar z}-2i\epsilon_1 CJ\Phi_z.$ From (\ref{e:hsgrwu}), it follows that
\begin{equation}\label{e:protosimantiko1}
|G^+(\Phi_z,\hat{\Phi}_z)|^2=4|g_1(\phi_z,\phi_z)|^2=4|g_2(\psi_z,\psi_z)|^2=\frac{1+4\epsilon_1 C^2}{4}.
\end{equation}
It is not hard for one to obtain,
\begin{equation}\label{e:protosimantiko2}
|G^+(H,J\Phi_z)|^2=-\frac{|H|^2}{4}.
\end{equation}
Since the normal curvature of $\Phi$ vanishes, the Ricci equation yields,
\begin{equation}\label{e:protosimantiko3}
|G^+(\Phi_{zz},J\Phi_z)|^2=-\frac{e^{6u}(|H|^2+4C^2)}{16}.
\end{equation}
We now use the Gauss equation to get $C^2=2e^{-4u}(|\Phi_{zz}|^2-|\Phi_{z\bar z}|^2)$,
which implies,
\[
\frac{C^2e^{4u}}{2}=
2|u_z|^2e^{2u}-
\frac{e^{2u}|G^+(H,J\Phi_z)|^2}{2}
-2e^{-2u}|G^+(\Phi_{zz},J\Phi_z)|^2
\]
\[
\qquad\qquad\qquad\qquad\qquad\qquad\qquad\qquad\qquad +\frac{|G^+(\Phi_z,\hat{\Phi}_z)|^2}{2}-\frac{e^{4u} (|H|^2+1/2)}{4},
\]
and from (\ref{e:protosimantiko1}), (\ref{e:protosimantiko2}) and (\ref{e:protosimantiko3}) we finally obtain $C^2=-4\epsilon_1|u_z|^2e^{-2u}$.
Since $g$ is flat, it is possible to choose local coordinates $(x,y)$ such that $g=dx^2+\epsilon_1 dy^2$, that is, the function $u$ is constant and therefore $C=0$. Then $\Phi$ is of projected rank one and in particular, it is locally a product of curves with constant curvature such that they cannot be both geodesics. 
$\Box$

\vspace{0.2in}

\section{($H$-) Stability of ($H$-) minimal Lagrangian surfaces}\label{s:hamiltstabilitysection}

For the stability of $G^\epsilon$-minimal Lagrangian surfaces in the para-K\"ahler structure $(\Sigma_1\times\Sigma_2,G^\epsilon,J,\Omega^\epsilon)$, we prove the following theorem:

\vspace{0.1in}

{\indent\sc Proof of Theorem \ref{t:stabilityofminimal}:}
Let $\Phi:S\rightarrow\Sigma_1\times\Sigma_2$ be a $G^{\epsilon}$-minimal Lagrangian surface. By assumption, $\Phi$ is of projected rank one and therefore it is locally the product $\gamma_1\times\gamma_2$ of non-null geodesics parametrised by $\Phi(s,t)=(\phi(s),\psi(t))$. If $(S_t)$ is a normal variation of $S$ with velocity $X\in{\mbox N}S$, the second variation formula is,
\[
\delta^2 V(S)(X)=\int_{S} \Big(G^{\epsilon}(\nabla^{\bot}X,\nabla^{\bot}X)-G^{\epsilon}(A_X,A_X)+G^{\epsilon}(R^{\bot}(X),X)\Big)dV,
\]
where, $A_X$ is the shape operator and $R^{\bot}(X):={\mbox Tr}\Big((Y_1,Y_2)\mapsto R(Y_1,X) Y_2)\Big)$.

For $X=X^1J\Phi_s+X^2J\Phi_t$, we have, $$G^{\epsilon}(\nabla^{\bot}X,\nabla^{\bot}X)= -(X^1_s)^2-(X^2_t)^2-\epsilon\epsilon_{\phi}\epsilon_{\psi}\big((X^2_s)^2+(X^1_t)^2\big)$$.
Furthermore, $G^{\epsilon}(A_X,A_X)=0$ and a brief computation gives,
\[
G^{\epsilon}(R^{\bot}(X),X)=\epsilon_{\phi}(X^1)^2\kappa(g_1)+\epsilon_{\psi}(X^2)^2\kappa(g_2).
\]
The metric $G^{\epsilon}$ is of neutral signature and therefore a necessary condition for a minimal surface to be stable is that the induced metric $\Phi^{\ast}G^\epsilon$ must be Riemannian \cite{An2}. This implies that $\epsilon_{\phi}=\epsilon\epsilon_{\psi}$ and hence the second variation formula becomes
\[
\delta^2 V(S)(X)=\int_{S}
-(X^1_s)^2-(X^2_s)^2-(X^1_t)^2-(X^2_t)^2+\epsilon_{\phi}(X^1)^2\kappa(g_1)+\epsilon_{\psi}(X^2)^2\kappa(g_2),
\]
and the theorem follows.
$\Box$

\begin{example}We give the following examples of $G^\epsilon$-minimal Lagrangian surfaces that are stable:

\bigskip

 \hspace{1em}  \begin{tabular}{| l || l ||  }
     \hline
     Product &  Type of $\gamma_1\times\gamma_2$  \\ \hline \hline
   $({\mbox  d}{\mathbb S}_a^2\times {\mbox  d}{\mathbb S}_b^2 ,G^+)$ & $\gamma_1\subset {\mbox  d}{\mathbb S}_a^2,\gamma_2\subset {\mbox  d}{\mathbb S}_b^2$ both timelike \\ \hline 
  $({\mbox  d}{\mathbb S}_a^2\times {\mbox  Ad}{\mathbb S}_b^2,G^-)$ & $\gamma_1\subset {\mbox  d}{\mathbb S}_a^2$ timelike and $\gamma_2\subset {\mbox  Ad}{\mathbb S}_b^2$ spacelike  \\  \hline 
  $ ({\mbox  d}{\mathbb S}_a^2\times {\mbox  Ad}{\mathbb S}_a^2,G^-)$ & $\gamma_1\subset {\mbox d}{\mathbb S}_a^2$ timelike and $\gamma_2\subset {\mbox  Ad}{\mathbb S}_a^2$ spacelike  \\ \hline 
   $({\mathbb D}\times {\mbox d}{\mathbb S}_b^2,G^+)$&  $\gamma_1\subset {\mathbb D},\gamma_2\subset {\mbox d}{\mathbb S}_b^2$ both timelike \\ \hline 
 $({\mathbb D}\times {\mbox d}{\mathbb S}_b^2,G^-)$ &  $\gamma_1\subset {\mathbb D}$ spacelike and $\gamma_2\subset {\mbox d}{\mathbb S}_b^2$ timelike  \\ \hline 
\end{tabular}
\end{example}

\bigskip

We now study the $H$-stability of a $H$-minimal surface in the para-K\"ahler structure $(\Sigma_1\times\Sigma_2,G^\epsilon,\Omega^\epsilon,J)$. We recall that the $H$-stability of a $H$-minimal surface $S$ in a pseudo-Riemannian manifold $(M, G)$ is given by the monotonicity of the second variation formula of the volume $V(S)$ under Hamiltonian deformations (see \cite{AnGer} and \cite{Oh1}). For a smooth compactly supported function  $u\in C^{\infty}_{c}(S)$ the second variation $\delta^2 V(S)(X)$ formula in the direction of the Hamiltonian vector field $X=J\nabla u$ is:
\begin{equation}\label{e:hamstb}
\delta^2 V(S)(X)=\int_{S} -(\Delta u)^2 + {\mbox Ric}^G(\nabla u,\nabla u)+2G(h(\nabla u,\nabla u),nH)+G^2(nH,J\nabla u),
\end{equation}
where $h$ is the second fundamental form of $S$, ${\mbox Ric}^G$ is the Ricci curvature tensor of the metric $G$, and $\Delta$ with $\nabla$ denote the Laplacian and gradient, respectively, with respect to the metric $G$ induced on $S$. For the Hamiltonian stability of projected rank one Hamiltonian $G^{\epsilon}$-minimal surfaces we prove the following theorem:

\vspace{0.1in}

{\indent\sc Proof of Theorem \ref{t:hstability}:}
Let $\Phi=(\phi,\psi):S\rightarrow \Sigma_1\times\Sigma_2$ be of projected rank one Hamiltonian $G^{\epsilon}$-minimal immersion and let $(s,t)$ be the corresponded arclengths of $\phi$ and $\psi$, respectively. 

After a brief computation, the second variation formula for the volume functional with respect to the Hamiltonian vector field $X=J\nabla u$ given by (\ref{e:hamstb}), becomes
\[
\delta^2 V(S)(X)=\int_{S} -(\epsilon_{\phi}u_{ss}+\epsilon\epsilon_{\psi} u_{tt})^2+\epsilon_{\phi}u_s^2\kappa_1+\epsilon_{\psi}u_t^2\kappa_2-(\epsilon_{\phi}u_s k_{\phi}-\epsilon\epsilon_{\psi} u_t k_{\psi})^2,
\]
where $\kappa_1,\kappa_2$ are the Gauss curvatures of $g_1$ and $g_2$, respetively and the theorem follows.
$\Box$

We also have the next Proposition:

\begin{Cor}
Let $(\Sigma,g)$ be a Riemannian two manifold of positive (negative) Gaussian curvature. Assume that every $G^-$-minimal Lagrangian surface $\Phi=(\phi,\psi)$ immersed in $\Sigma\times\Sigma$ is a pair of timelike (spacelike) geodesics in $\Sigma$. Then  $\Phi$ is $H$-stable.
\end{Cor} 

\begin{example} We know from \cite{An4} that every $G^-$-minimal Lagrangian immersion in ${\mbox Ad}{\mathbb S}^3$ is the Gauss map of a equidistant tube along a geodesic $\gamma$ in ${\mbox Ad}{\mathbb S}^3$ and following the example \ref{e:gaussmaporgeodesiccongruence} it must be locally parametrised as the product of geodesics in ${\mbox d}{\mathbb S}^2$. In this example, we are going to see exactly how can we obtain this product of geodesics. Assume that $\gamma$ is a spacelike geodesic. Since the metric $G^-$ is invariant under the natural action of the isometry group of ${\mbox Ad}{\mathbb S}^3$, we may assume that $\gamma$ is parametrised by $\gamma(s)=(0,0,\cos s,\sin s)$. In this case the tube over $\gamma$ with constant distance $d>0$ is parametrised by,
\[
f:{\mathbb S}^1\times{\mathbb S}^1\rightarrow {\mbox Ad}{\mathbb S}^3:(s,t)\mapsto (\sinh d\cos t,\sinh d\sin t,\cosh d\cos s,\cosh d\sin s).
\]
The normal vector field is $N(s,t)=(\cosh d\cos t,\cosh d\sin t,\sinh d\cos s,\sinh d\sin s)$, and thus $(f,v_1:=f_s/|f_s|,v_2:=f_t/|f_t|,N)\in SO(2,2)$. It is also known by \cite{An4}, that $ L^-({\mbox Ad}{\mathbb S}^3)$ is identified with the Grassmannian ${\mbox Gr}^-(2,4):=\{x\wedge y\in\Lambda^2({\mathbb R}^3):\; y\in T_x{\mbox Ad}{\mathbb S}^3,\; \left<y,y\right>_2=-1\}$. For an oriented orthonormal frame $\{e_1,e_2,e_3,e_4\}$ of ${\mathbb R}^{2,2}:=({\mathbb R}^4,\left<.,.\right>_2)$ such that $-|e_1|=-|e_2|=|e_3|=|e_4|=1$ we define the subspaces ${\mathbb R}_{\pm}^{1,2}$ of $\Lambda^2({\mathbb R}^3)$ generated by the vectors
\[
E^1_{\pm}=(e_1\wedge e_2\pm e_3\wedge e_4)/\sqrt{2},\quad E^2_{\pm}=(e_1\wedge e_3\pm e_4\wedge e_2)/\sqrt{2},\quad E^3_{\pm}=(e_1\wedge e_4\pm e_2\wedge e_3)/\sqrt{2}.
\]
We define the de Sitter 2-spaces by ${\mbox d}{\mathbb S}_\pm^2:=\{x=x_1E^1_{\pm}+x_2E^2_{\pm}+x_3E^3_{\pm}\in {\mathbb R}_{\pm}^{1,2}|\; \left<x,x\right>_2=-1\}.$
If $u_1\wedge u_2\in {\mbox Gr}^-(2,4)$ take $u_3,u_4$ such that $(u_1,u_2,u_3,u_4)\in SO(2,2)$. The map 
\[
{\mbox Gr}^-(2,4)\rightarrow {\mbox d}{\mathbb S}_-^2\times {\mbox d}{\mathbb S}_+^2:
u_1\wedge u_2\mapsto ((u_1\wedge u_2+u_3\wedge u_4)/\sqrt{2},(u_1\wedge u_2-u_3\wedge u_4)/\sqrt{2}),
\]
is a diffeomorphism. The Gauss map $\bar f=f\wedge N\in L^-({\mbox Ad}{\mathbb S}^3)$ is identified to the following pair $(\phi,\psi)\in {\mbox d}{\mathbb S}^2_-\times {\mbox d}{\mathbb S}^2_+$, given by
\[
\phi(s,t)=(f\wedge N+v_1\wedge v_2)/\sqrt{2},\quad
\psi(s,t)=(f\wedge N-v_1\wedge v_2)/\sqrt{2}.
\]
By setting $u=t-s$ and $v=t+s$, we observe that $\phi$ and $\psi$ are the following geodesics of ${\mbox d}{\mathbb S}^2_{\pm}$:
\[
\phi(u)=-\cos u E_-^2+\sin u E^3_-,\quad \psi(v)=-\cos (v) E_+^2-\sin (v) E^3_+.
\]
Firthermore, $\phi$ and $\psi$ are timelike geodesics and therefore $\bar f$ is a (unstable) $H$-stable minimal Lagrangian torus in $L^{-}(\mathop{\mbox Ad}{\mathbb S}^3)$. For the case where $\gamma$ is a timelike geodesic, a similar argument shows that $\bar f$ is $H$-unstable. We emphasize here that  by using different methods  in \cite{AnGer}, it was first proven that the Gauss map $\bar f$ is $H$-stable. 
\end{example}

\end{document}